\documentclass{amsart} 

\usepackage[dvips]{epsfig}
\usepackage{latexsym}
\usepackage{amsfonts,amsmath,amssymb}
\newtheorem{corollary}{Corollary}
\newtheorem{lemma}{Lemma}
\newtheorem{example}{Example}
\newtheorem{theorem}{Theorem}
\newtheorem{proposition}{Proposition}
\newtheorem{question}{Question}
\newtheorem{definition}{Definition}

\def\cal{\mathcal}
\newcommand{\Sh}{{\cal S}_H}
\begin{document}

\title{The Average Number of Block Interchanges Needed to Sort A Permutation
and a recent result of Stanley}
\author[Bona]{Mikl\'os B\'ona${}^1$}
\address{Department of Mathematics\\
  University of Florida\\
  Gainesville, FL 32611-8105}
\email{bona@math.ufl.edu}
\thanks{${}^1$Research supported by the National Science Foundation,
the National Security Agency, and the Howard Hughes Medical Institute.}

\author[Flynn]{Ryan Flynn${}^2$}
\address{Department of Mathematics\\
Penn State University\\
University Park, State College, PA 16802}
\email{flynn@math.psu.edu}
\thanks{${}^2$ 
Research supported by the Howard Hughes Medical Institute.}

\begin{abstract} We use an interesting result of probabilistic flavor
concerning the product of two permutations consisting of one cycle each 
to find an explicit formula for the average number of block interchanges
needed to sort a permutation of length $n$.
\end{abstract}

\maketitle

\section{Introduction}
\subsection{The main definition, and the outline of this paper}
Let $p=p_1p_2\cdots p_n$ be a permutation. A {\em block interchange}
is an operation that interchanges two blocks of consecutive entries
 without
changing the order of entries within each block. The two blocks do not need
to be adjacent. Interchanging the blocks $p_{i}p_{i+1}\cdots p_{i+a}$
and $p_{j}p_{j+1}\cdots p_{j+b}$ with $i+a<p_j$ results in the permutation
\[p_1p_2\cdots p_{i-1}p_{j}p_{j+1}\cdots p_{j+b}p_{i+a+1}
\cdots p_{j-1}p_ip_{i+1}
\cdots p_{i+a}p_{j+b+1}\cdots p_n.\]
For instance, if $p=3417562$, then interchanging the block of the first
two entries with the block of the last three entries results in the 
permutation $5621734$. 

In this paper, we are going to compute the average number of block
 interchanges
to sort a permutation of length $n$.
The methods used in the proof are surprising for several reasons. 
First, our enumeration problem will lead us to an interesting question
on the symmetric group that is very easy to ask and that is of probabilistic
flavor. Second, this question then turns out to be surprisingly difficult to
answer-- the conjectured answer of one of the authors has only recently been
proved by Richard Stanley \cite{stanley}, whose proof was not elementary. 

\subsection{Earlier Results and Further Definitions}

The first significant result  on the topic of sorting by block 
interchanges is by D. A. Christie \cite{christie}, who provided a
direct way of determining the number of block interchanges necessary
to sort any given permutation $p$. The following definition was crucial to
his results. 
\begin{definition} 
 The {\em cycle graph} $G(p)$ of the permutation $p=p_1p_2\cdots p_n$
 is a directed graph on
vertex set $\{0,1,\cdots ,n\}$ and  $2n$ edges that are colored either
black or gray as follows. Set $p_0=0$.
\begin{enumerate}
\item For $0\leq i\leq n$, there is a black edge from $p_i$ to $p_{i-1}$, 
where the indices are to be read modulo $n+1$, and
\item For $0\leq i\leq n$, there is a gray edge from $i$ to $i+1$, where
the indices are to be read modulo $n+1$.
\end{enumerate}
\end{definition}

See Figure \ref{christiegraphs} for two examples.

\begin{figure}[ht]
 \begin{center}
  \epsfig{file=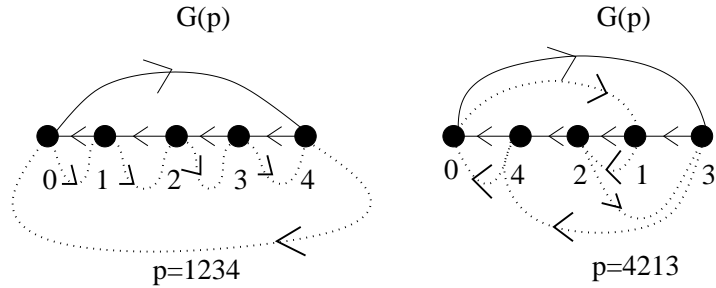}
  \label{christiegraphs}
\caption{The graphs $G(p)$ for $p=1234$ and $p=4213$.}
 \end{center}
\end{figure}

It is straightforward to show that $G(p)$ has a unique decomposition into
edge-disjoint directed cycles in which the colors of the edges alternate. 
Let $c(G(p))$ be the number of directed
cycles in this decomposition of $G(p)$.  
The main enumerative result of \cite{christie} is the following 
formula. In the rest of this paper,  permutations of length $n$ will be
called $n$-permutations, for shortness.

\begin{theorem} \label{christie}
The number of block interchanges needed to sort the
$n$-permutation $p$ is $\frac{n+1-c(G(p))}{2}$.
\end{theorem}

Note that in particular this implies that $n+1$ and $c(G(p))$ are always
of the same parity. Christie has also provided an algorithm that sorts
$p$ using  $\frac{n+1-c(G(p))}{2}$ block interchanges. As the identity
permutation is the only $n$-permutation that takes zero block interchanges
to sort, it is the only $n$-permutation $p$ satisfying $c(G(p))=n+1$.

Theorem \ref{christie} shows that
 in order to find the average number $a_n$ of block interchanges
needed to sort an $n$-permutation, we will need the average value of $c(G(p))$
for such permutations. The following definition \cite{doignon} will be useful.

\begin{definition} The {\em Hultman number} $\Sh(n,k)$ is the number of
$n$-permutations $p$ satisfying $c(G(p))=k$. 
\end{definition}

So the Hultman numbers are somewhat analogous to the signless Stirling numbers
of the first kind that count $n$-permutations with $k$ cycles. 

This is a good place to point out that in this paper, we will sometimes
discuss {\em cycles of the permutation} $p$ in the traditional sense, 
which are not to be confused with
 the {\em directed cycles of} $G(p)$, counted by $c(G(p))$. 
Following \cite{doignon}, the number of cycles of the permutation $p$ will
be denoted by $c(\Gamma(s))$. Indeed, the cycles of a permutation $p$ are
equivalent to the directed cycles of the graph $\Gamma(p)$ in which there is
an edge from $i$ to $j$ if $p(i)=j$. 
For instance, if $p=1234$, then 
$c(\Gamma(p))=4$, while $c(G(p))=5$.

The following recent theorem of Doignon and Labarre \cite{doignon}
brings the Hultman numbers closer to the topic of enumerating permutations
according to their cycle structure (in the traditional sense).
Let $S_n$ denote the symmetric group of degree $n$. 

\begin{theorem} \label{doignon}
  The  Hultman number $\Sh(n,k)$ is equal to the number
of ways to obtain the cycle $(12..\cdots n(n+1))\in S_{n+1}$
 as a product  $qr$
of permutations, where $q\in S_{n+1}$ is any cycle of length $n+1$,
 and the permutation $r\in S_{n+1}$ has exactly $k$ cycles, that is
$c(\Gamma(r))=k$. 
\end{theorem} 

\section{Our Main Result}
The following immediate consequence of Theorem \ref{doignon} is more suitable
for our purposes.

\begin{corollary}
 The  Hultman number $\Sh(n,k)$ is equal to the number of $(n+1)$-cycles
$q$ so that the product $(12\cdots n(n+1))q$ is a permutation
with exactly $k$ cycles, that is, $c(\Gamma((12\cdots n(n+1))q)=k$.
\end{corollary}

\begin{example} For any fixed $n$, we have $\Sh(n,n+1)=1$ since $c(G(p))=n+1$
if and only if $p$ is the identity permutation. And indeed, there is
exactly one $(n+1)$-cycle (in fact, one permutation)
 $q\in S_{n+1}$ so that  $(12..\cdots n(n+1))q$ has $n+1$ cycles, 
namely $q=(12\cdots n(n+1))^{-1}=(1(n+1)n\cdots 2)$. 
\end{example}

In other words, finding the average of the numbers $c(G(p))$ over 
all $n$-permutations $p$ is equivalent
to finding the average of the numbers
$c(\Gamma((12..\cdots n(n+1))q)$, where $q$ is an $(n+1)$-cycle.

Let us consider the product $s=(12\cdots n)z$, where $z$ is a cycle of length
$n$. Let us insert the entry $n+1$ into $z$ to get the permutation $z'$ 
so that $n+1$ is inserted
between two specific entries $a$ and $b$ in the following sense.
\[z'(i)=\left\{ \begin{array}{l@{\ }l}
z(i) \hbox{ if $i\notin \{a,n+1\}$},\\
n+1 \hbox{ if $i=a$, and}\\
b \hbox{ if $i=n+1$.}
\end{array}\right.
\]
See Figure \ref{insert} for an illustration.
\begin{figure}[ht]
 \begin{center}
  \epsfig{file=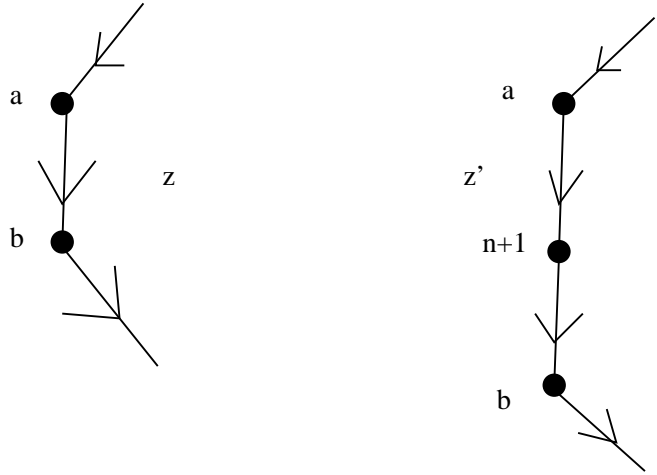}
  \label{insert}
\caption{How $z'$ is obtained from $z$.}
 \end{center}
\end{figure}

The following proposition is the first step towards describing how the
Hultman numbers grow. 

\begin{proposition} \label{equal}
Let $a$, $b$, and $z'$ be defined as above, and let 
$s'=(12\cdots (n+1)z'$. Then we have
 \[c(\Gamma(s' ) = \left\{ \begin{array}{l@{\ }l}
c(\Gamma(s)) -1 \hbox{ if $2\leq a$, and 
$a-1$ and $z(1)$ are not in the same cycle
of $s$},\\
c(\Gamma(s))+1 \hbox{ if $2\leq a$, and   $a-1$ and $z(1)$ are in 
the same cycle
of $s$, and}\\ c(\Gamma(s))+1 \hbox{ if $a=1$}.
\end{array}\right.
\]
\end{proposition}

\begin{proof}
Let us assume first that $a\geq 2$, and that $a-1$ is in a cycle $C_1$
of $s$, and $z(1)$ is in a different cycle $C_2$ of $s$. Let 
$C_1=((a-1)b\cdots )$ and let $C_2=(z(1)\cdots n)$. After the insertion
of $n+1$ into $z$, the newly obtained permutation $s'=(12..\cdots (n+1)z'$
sends $a-1$ to $n+1$, then $n+1$ to $z(1)$, then leaves the rest of 
$C_2$ unchanged till its last entry. Then it sends $n$ back to 
$z'(n+1)=b$, from where it continues with the rest of $C_1$ with no change.
So in $s'$, the cycles $C_1$ and $C_2$ are united, the entry $n+1$ joins
their union, and there is no change to the other cycles of $s$.
See Figure \ref{firstcase} for an illustration.

\begin{figure}[ht]
 \begin{center}
  \epsfig{file=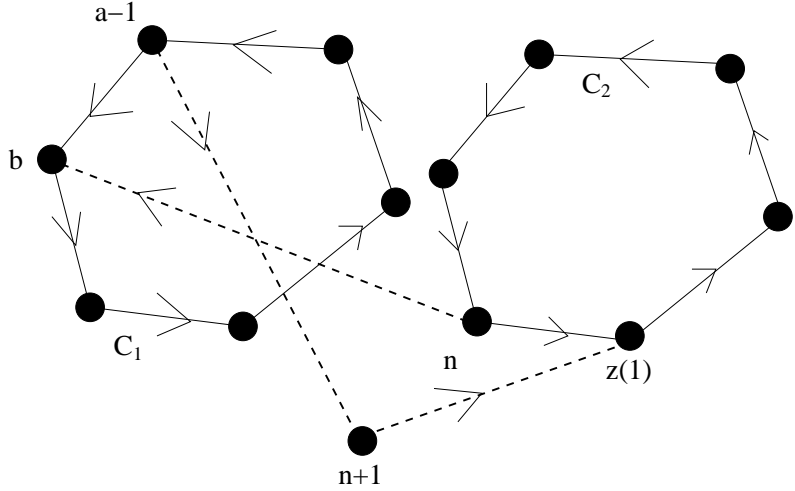}
  \label{firstcase}
\caption{If $a-1$ and $z(1)$ are in different cycles of $s$, those 
cycles will turn into one.}
 \end{center}
\end{figure}

Let us now assume that  $a\geq 2$, and that $a-1$ and $z(1)$ are both
in the same cycle $C$ of $s$. Then $C=((a-1)b\cdots n(z(1))\cdots )$.
After the insertion
of $n+1$ into $z$, the newly obtained permutation $s'=(12..\cdots (n+1)z'$
sends $a-1$ to $n+1$, then $n+1$ to $z(1)$, cutting off the part of
$C$ that was between $a-1$ and $n$. So $C$ is split into two cycles, 
the cycle $C'=((a-1)(n+1)z(1)\cdots )$ and the cycle $C''=(b\cdots n)$. 
Note that $s'(n)=b$ since $z'(n+1)=b$. See Figure \ref{secondcase}
 for an illustration.

\begin{figure}[ht]
 \begin{center}
  \epsfig{file=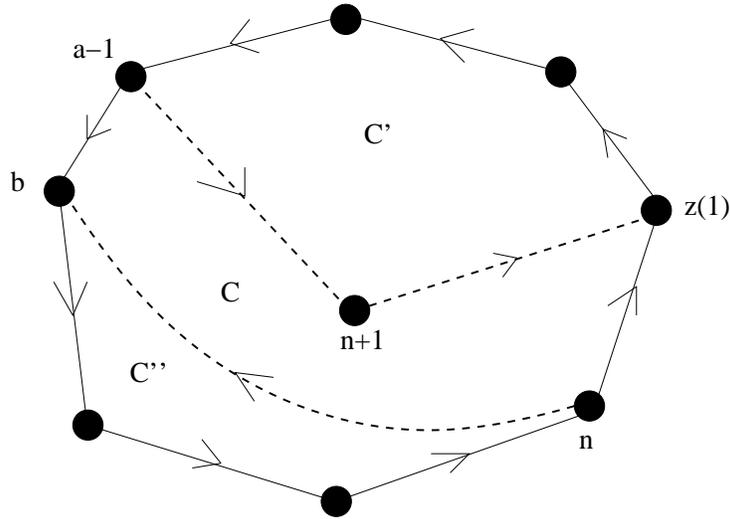}
  \label{secondcase}
\caption{If $a-1$ and $z(1)$ are in the same cycle of $s$, that 
cycle will split into two cycles.}
 \end{center}
\end{figure}

Finally, if $a=1$, then $s'(n+1)=(n+1)$, and the rest of the cycles
of $s$ do not change. 
\end{proof}

Let $TC_n$ denote the set of $n$-permutations  that can be
obtained as a product of two cycles of length $n$. Let
$a_n$ be the average number of cycles of the elements of $TC_n$.
 
\begin{lemma} \label{oddcase}
For all positive integers $m$, 
we have $a_{2m+2}=a_{2m+1}+\frac{1}{2m+1}$. 
\end{lemma}

\begin{proof}
We apply Proposition \ref{equal}, with $n$ replaced by $n+1=2m+1$. This
means $n+1$ is an odd number. So $z'$
is a cycle of length $2m+2$ obtained from a cycle $z$ of 
length $2m+1$ through the insertion of the maximal element $2m+2$ into 
one of $2m+1$ possible positions. By Proposition  \ref{equal}, some of these
insertions increase $c(G(s))$ by one, and others decrease $c(G(s))$ by one,
depending on whether $a-1$ and $z(1)$ are in the same cycle of $s$ or not.
The question is, of course, how many times they will be in the same cycle of 
$s$. 

This question is easily seen to be equivalent to the following question.

\begin{question} Let $i$ and $j$ be two fixed elements of the 
set $\{1,2,\cdots ,h\}$
Select an element $p$ of $TC_h$ at random. What is the probability that
 $p$ contains $i$ and $j$ in the same cycle?
\end{question}

The first author of this article has conjectured that the answer to this
question was 1/2 for odd $h$. This conjecture was recently
proved by Richard Stanley \cite{stanley},
 who also settled the question for even values
of $h$.

\begin{theorem} \label{stanley} \cite{stanley} 
Let $i$ and $j$ be two fixed, distinct
 elements of the set $[h]=\{1,2,\cdots ,h\}$, where $h>1$.
Let $x$ and $y$ be two randomly selected $h$-cycles over $[h]$.
Let $p(h)$ be the probability that $i$ and $j$ are in the same cycle of
$xy$. Then 
\[p(h)= \left\{ \begin{array}{l@{\ }l} 
\frac{1}{2} \hbox{ $if$ $h$ is odd, and }\\
\frac{1}{2}-\frac{2}{(h-1)(h+2)}  \hbox{ $if$ $h$ is even}.
\end{array}\right.
\]
\end{theorem}

The  proof of Lemma \ref{oddcase} is now straightforward. If $a\neq 1$,
then $a-1$ and $z(1)$ are equally likely to be in the same cycle or not in
the same cycle of $s$. Therefore, an increase of one or a decrease of
one in $c(G(s))$ is equally likely. If, on the other hand, $a=1$, which
occurs in $1/(2m+1)$ of all cases, then $c(G(s))$ increases by one. 
So \[a_{2m+2}=\frac{2m}{2m+1} a_{2m+1} +    \frac{1}{2m+1}(a_{2m+1}+1)
=a_{2m+1}+   \frac{1}{2m+1}.\]
\end{proof}

Note that the statement of Lemma \ref{oddcase} holds even when $m=0$, since
$a_2=2=a_1+1$.

As Theorem \ref{stanley} provided a formula for $P(h)$ for even values of
$h$ as well as odd values, we can state and prove the analogous version of
Lemma \ref{oddcase} for the integers not covered there.
 
\begin{lemma} \label{evencase}
For all positive integers $m$, we have
\[a_{2m+1}=a_{2m}+\frac{m}{2(m+1)(m+2)}=a_{2m}+\frac{1}{2m}-
\frac{1}{m(m+1)}.\]
\end{lemma}

\begin{proof}
This is very similar to the proof of Lemma \ref{oddcase}.
If $a\neq 1$, which happens in $(2m-1)/(2m)$ of all cases, then
the probability of $a-1$ and $z(1)$ falling into the same cycle of $s$
 is $\frac{1}{2}-\frac{2}{(2m+1)(2m+4)}$ by Theorem
\ref{stanley}. By Proposition \ref{equal}, in these cases $c(G(s))$ grows
by one.  If $a=1$, which occurs in $1/(2m+2)$ of all cases, $c(G(s))$
always grows by one. 
So 
\begin{eqnarray*} a_{2m+1} & = & \frac{2m-1}{2m}\cdot
\left(\frac{1}{2}-\frac{2}{(2m-1)(2m+2)}\right)
(a_{2m}+1) \\
& + & \frac{2m-1}{2m}\cdot\left(\frac{1}{2}+
\frac{2}{(2m-1)(2m+2)}\right)
(a_{2m}-1) \\
& + & \frac{1}{2m} (a_{2m}+1),\end{eqnarray*}
which is equivalent to the statement of the lemma as can be seen after 
 routine rearrangements.
\end{proof}

We are now in position to state and prove our formula for the average number
$a_n$ of cycles in elements of $TC_n$.

\begin{theorem} \label{average}
We have $a_1=1$, and 
\[a_n= \frac{1}{\lfloor (n-1)/2 \rfloor +1} + 
\sum_{i=1}^{n-1} \frac{1}{i}. \]
\end{theorem}

Note that this formula produces even the correct value even for $n=1$, 
that is, it produces the equality $a_1=1$.

\begin{proof} (of Theorem \ref{average})
 The statement is now a direct consequence of Lemmas 
\ref{oddcase} and \ref{evencase} if we note the telescoping sum
$\sum_{i=1}^t \frac{1}{i(i+1)} = 1-\frac{1}{t+1}$ obtained when
summing the values computed in  
Lemma \ref{evencase}. 
\end{proof}

Note that it is well-known that on  average, an $n$-permutation has
$\sum_{i=1}^n \frac{1}{i}$ cycles. This is the average value of $c(\Gamma(p))$
for a randomly selected $n$-permutation. Theorem  \ref{average} shows
that the average value of $c(G(p))$ differs from this by about $1/n$. 

Finally, our main goal is easy to achieve.

\begin{theorem}
The average number of  block interchanges needed to sort an $n$-permutation
is \[b_n=\frac{n- \frac{1}{\lfloor n/2 \rfloor +1}-
\sum_{i=2}^{n} \frac{1}{i}}{2}.\]
\end{theorem}

\begin{proof} By Theorem \ref{doignon} and Theorem \ref{average},
the average number value of $c(G(p))$ over all permutations $p$ of length
$n$ is $a_{n+1}= \frac{1}{\lfloor n/2 \rfloor +1} + 
\sum_{i=1}^{n} \frac{1}{i}$.
Our claim now immediately follows from Theorem \ref{christie}.
\end{proof}

So the average number of block interchanges needed to sort an
$n$-permutation is close to $(n-\log n)/2$. 

\section{Remarks and Further Directions}
Richard Stanley's proof of Theorem \ref{stanley} is not elementary.
It uses symmetric functions, exponential generating functions, integrals,
and a formula of Boccara \cite{boccara}. A more combinatorial proof of
the stunningly simple answer for the case of odd $k$ would still be 
interesting.

As pointed out by Richard Stanley \cite{stanleyp},
there is an alternative way to obtain the result of Theorem \ref{average} 
without
using Theorem \ref{stanley}, but that proof in turn uses symmetric functions
and related machinery. It is shown in Exercises 69(a) and 69(c) of
\cite{stanleyex} that
\begin{equation}
\label{stirling} 
P_n(q)=\sum_{p\in TC_n}q^{c(\Gamma(p))}=\frac{1}{{n+1\choose 2}}
\sum_{i=0}^{\lfloor (n-1)/2 \rfloor} c(n+1,n-2i)q^{n-2i},\end{equation}
where, as usual, $c(n,k)$ is a signless Stirling number of the first kind,
that is, the number of permutations of length $n$ with $k$ cycles.
Now $a_n$ can be computed by considering $P_n'(1)$, which in turn 
can be computed by considering the well-known identity
\[F_{n+1}(x)=\sum_{k=1}^{n+1}c(n+1,k)x^k=x(x+1)\cdots (x+n),\]
and then evaluating $F_{n+1}'(1)+F_{n+1}'(-1)$.

The present
 paper provides further evidence that the cycles of the graph $G(p)$
have various enumerative properties that are similar to the enumerative
properties of the graph $\Gamma(p)$, that is, the cycles of the permutation
$p$. This raises the question as to which well-known properties of the
Stirling numbers, such as unimodality, log-concavity, real zeros property,
 hold for the Hultman numbers as well. (See for instance Chapter 8 of
\cite{intro} for definitions and basic information on these properties.) 
 A simple  
 modification is necessary since  $\Sh(n,k)=0$ if 
$n$ and $k$ are of the same parity.
So let
\[Q_n(q)= \left\{ \begin{array}{l@{\ }l}\sum_{p\in TC_n}q^{c(\Gamma(p))/2}
\hbox{ if $n$ is even}\\
\sum_{p\in TC_n}q^{(c(\Gamma(p))+1)/2} \hbox{ if $n$ is odd.}\end{array}
\right. \]

While the coefficients of $P_n(q)$ are {\em all} the Hultman numbers
$\Sh(n-1,1)$, \linebreak $\Sh(n-1,2),\cdots, \Sh(n-1,n-1)$,
 the coefficients of $Q_n(q)$ are the
{\em nonzero} Hultman numbers $\Sh(n-1,k)$. 

Clearly, $Q_n(q)=P_n(q^2)$ if $n$ is even, and $Q_n(q)=qP_n(q^2)$
if $n$ is odd. However, Exercise 69(b) of \cite{stanleyex} shows that
all roots of $P_n(q)$ have real part 0. Hence the roots of $Q_n(q)$ are
all real and non-positive, from which the log-concavity and unimodality
of the coefficients of $Q_n(q)$ follows.  
This raises the question of whether there is a combinatorial proof
for the latter properties, possibly along the lines of the work of 
Bruce Sagan
 \cite{Sagan88} for the Stirling numbers of both kinds. 
Perhaps it is useful to note that
(\ref{stirling})
and Theorem \ref{doignon} imply that 
\[\Sh(n,k)=\left\{ \begin{array}{l@{\ }l}
c(n+2,k)/\binom{n+2}{2} \hbox{ if $n-k$ is odd,} \\
0 \hbox{ if $n-k$ is even.}\end{array}\right.
\]

Finally, to generalize in another direction, we point out that it is
very well-known (see, for example, Chapter 4 of \cite{intro}),
 that if we select a $n$-permutation $p$ at random, and 
$i$ and $j$ are two fixed, distinct positive integers at most as large as
$n$, then  the probability that $p$ contains $i$ and $j$ in the same cycle 
is $1/2$. Theorem \ref{stanley} shows that if $n$ is odd, then the set
$TC_n$ behaves just like the set $S_n$ of all permutations in this aspect.
This raises the question whether there are other naturally defined subsets
of $n$-permutations in which this phenomenon occurs.  

\vskip 2 cm
\centerline{{\bf Acknowledgment}}
We are indebted to Richard Stanley for helpful discussions on various
aspects of Exercises 69(a-c) of \cite{stanleyex}, 
and, most of all, for proving Theorem \ref{stanley}. We are also
thankful to Anthony Labarre and Axel Hultman for bringing earlier results
to our attention.

\end{document}